\documentclass[11pt]{article}
\usepackage[margin=1in]{geometry}
\usepackage[utf8]{inputenc}
\usepackage{amsmath,amsfonts,amssymb,amsthm}
\usepackage{mathrsfs}
\usepackage[dvipsnames]{xcolor}
\usepackage{mathtools}
\usepackage{graphicx}
\usepackage{tikz-cd}
\usepackage[hidelinks]{hyperref} 
\usepackage[shortlabels]{enumitem}
\usepackage{bbm}
\usepackage{float}

\setlist{nolistsep}

\newcommand{\alg}[1]{\mathscr{#1}}
\newcommand{\cat}[1]{\mathcal{#1}}
\newcommand{\vir}[1]{\mathfrak{#1}}
\newcommand{\drm}{\mathrm{d}}

\DeclareMathOperator{\ns}{ns}
\DeclareMathOperator{\reduce}{red}
\DeclareMathOperator{\GL}{GL}
\DeclareMathOperator{\id}{id}
\DeclareMathOperator{\Det}{Det}
\DeclareMathOperator{\Ker}{Ker}
\DeclareMathOperator{\Hom}{Hom}

\let\det\relax 
\DeclareMathOperator{\det}{det}

\newtheorem{thm}{Theorem}[section]
\newtheorem{lem}[thm]{Lemma}
\newtheorem{prp}[thm]{Proposition}
\newtheorem{cor}{Corollary}[thm]

\theoremstyle{definition}
\newtheorem{dfn}{Definition}[section]

\theoremstyle{remark}

\title{$L^2$-torsion of fibrations}
\author{Chengzhang Sun}
\date{}

\begin{document}
    \pagenumbering{roman}
    \maketitle
    \begin{abstract}
        The paper studies the $L^2$-torsion of fibrations, focusing on cases that relax acyclicity and the determinant class condition. We prove the sum formula and the product formula for $L^2$-torsion in the extended abelian category. The desired formula for $L^2$-torsion of a simple fibration is obtained under the assumption that the fibers have zero Euler characteristic. 
    \end{abstract}
    \pagenumbering{arabic}
    \section{Introduction}
    In 1976, Atiyah \cite{atiyah1976elliptic} proposed to study $L^2$-invariants of noncompact manifold and he defined $L^2$-Betti numbers analytically. Since then, various $L^2$-invariants have been introduced, and among those, $L^2$-torsion was introduced around 1992 as a real-valued invariant of cocompact $G$-manifolds \cite{carey1992l2,mathai1992l2,lott1992heat}. It is a generalization the Reidemeister torsion and the analytic torsion of compact manifolds. There is a technical condition known as the determinant class condition. Besides, in the combinatorial setting, one must assume acyclicity; this condition can be avoided if one define the torsion as an element of a determinant line \cite{farber1997determinant}. 
    
    It was conjectured that any cocompact $G$-manifold is of determinant class \cite{lueck2002l2}, and it has been known for a large class of groups \cite{elek2005hyperlinearity}. If the Novikov-Shubin invariant is positive, then the determinant class condition holds. However, recent work of Austin \cite{austin2013rational} and Grabowski \cite{grabowski2016irrational} shows that the Novikov-Shubin invariant may be any nonnegative real number. It remains unclear if the determinant class condition is always satisfied. This condition was removed in \cite{braverman2005l2}, by working in an extended abelian category. The process is outlined in \S2 of this paper, which is also known as the combinatorial $L^2$-torsion. The definition of $L^2$-torsion in this paper is slightly more general than the original one. There is also an analytic $L^2$-torsion defined in an extended abelian category, and it is known that it can be identified with the combinatorial $L^2$-torsion \cite{braverman2005l2,zhang2005extended}, generalizing the Cheeger-M{\"u}ller theorem. 
    
    In this paper, we establish the sum formula and the product formula in \S3, generalizing the corresponding formulas in the classical case. In \S4, we use these results to study the $L^2$-torsion of fibrations. In \S5, some open problems are mentioned. 

    \paragraph{Acknowledgement.} The author wishes to thank Professor Xianzhe Dai for his guidance and reading through the draft of the paper. 

    \section{Rudiments of \texorpdfstring{$L^2$}{L2}-torsion}

    For more information and a more detailed discussion, see \cite{braverman2005l2}. 

    \subsection{Setup}
    Let $\alg{A}=(\alg{A},\tau)$ be a finite von Neumann algebra with a fixed choice of trace $\tau$. For example, let $G$ be a discrete group, and take $\alg{A}$ to be the group von Neumann algebra, and take $\tau$ to be the von Neumann trace. 

    Let $l^2(\alg{A})$ be the completion of $\alg{A}$ with respect to the inner product 
    \[(a,b)=\tau(b^*a).\] 
    A (left) \emph{Hilbert $\alg{A}$-module} is a Hilbert space $V$ with a continuous $\alg{A}$-action such that there is an equivariant embedding 
    \[V\hookrightarrow l^2(\alg{A})\otimes H\]
    for some Hilbert space $H$. We say that $V$ is \emph{finitely generated} if $H$ can be taken to be finite dimensional. 

    A \emph{Hilbertian $\alg{A}$-module} is a topological vector space $V$ with a continuous $\alg{A}$-action such that there exists an inner product on $V$ that makes $V$ a Hilbert $\alg{A}$-module with respect to the aforementioned $\alg{A}$-action. Any such inner product is called \emph{admissible}. Different choices of admissible inner products give rise to isomorphic Hilbert $\alg{A}$-modules. Indeed, if $\langle\ ,\ \rangle$ and $\langle\ ,\ \rangle'$ are admissible, then 
    \begin{equation}\label{eq:ip}
        \langle x,y\rangle'=\langle \alpha(x),y\rangle
    \end{equation}
    for some positive operator $\alpha:V\to V$. 

    The category of finitely generated Hilbertian $\alg{A}$-modules is denoted as $\cat{H}(\alg{A})$. The morphisms are equivariant continuous linear maps. It is a preabelian category. The \emph{extended abelian category} in the sense of Freyd is denoted as $\cat{E}(\alg{A})$. An object in $\cat{E}(\alg{A})$ is a morphism $\alpha:A'\to A$ in $\cat{H}(\alg{A})$, and a morphism from $\alpha:A'\to A$ to $\beta:B'\to B$ is an equivalence class of morphisms $f:A\to B$ such that $\exists g:A'\to B'$ with $f\alpha=\beta g$, where $f\sim f'$ if $f-f'=\beta F$ for some $F:A\to B'$. 

    The objects of $\cat{E}(\alg{A})$ are denoted as $\vir{X},\vir{Y},\vir{Z},\ldots$. Every element $\vir{X}$ is isomorphic to $A'\xrightarrow{\alpha}A$ for some injective $\alpha$ (by quotient out the kernel). 

    There is an embedding $F:\cat{H}(\alg{A})\to\cat{E}(\alg{A})$, $F(A)=(0\to A)$. It is fully faithful and allows us to identify objects in $\cat{H}(\alg{A})$ as projective objects in $\cat{E}(\alg{A})$. We say that $\vir{X}\in\cat{E}(\alg{A})$ is \emph{torsion} if $\vir{X}=(A'\xrightarrow{\alpha} A)$ where $\alpha$ has dense image. The full subcategory of torsion objects is denoted as $\cat{T}(\alg{A})$. For torsion $\vir{X}$, there is an associated invariant $\ns(\vir{X})\in [0,\infty]$ known as the \emph{Novikov-Shubin invariant}. 

    Any $\vir{X}\in \cat{E}(\alg{A})$ can be decomposed into its projective part and torsion part, in the sense that there exists a short exact sequence
    \[0\to T\vir{X}\to \vir{X}\to P\vir{X}\to 0.\]
    The short exact sequence splits but the splitting is not canonical. This defines two functors 
    \[T:\cat{E}(\alg{A})\to\cat{T}(\alg{A}), \quad P:\cat{E}(\alg{A})\to\cat{H}(\alg{A}).\]

    A \emph{projective} cochain complex in $\cat{E}(\alg{A})$ is a cochain complex $C=(C^\bullet,d)$ such that $C^\bullet$ is a projective object, i.e., $C^\bullet\in\cat{H}(\alg{A})$. We always assume that the complex has finite length. Computing the cohomology in $\cat{E}(\alg{A})$, we get \emph{$i$-th extended $L^2$-cohomology}
    \[H^i(C)=(C^{i-1}\xrightarrow{d}\Ker d|_{C^i}).\]
    Its projective part is the \emph{$i$-th reduced $L^2$-cohomology} $H_{\reduce}^i(C)$, and its torsion part gives rise to $\ns(H^i(C))$, which is the \emph{$i$-th Novikov-Shubin invariant} $\ns^i(C)$. We always compute cohomology in the extended abelian category unless the contrary is explicitly stated. 

    \subsection{Determinant line and the torsion}

    Let $A\in \cat{H}(\alg{A})$. Denote 
    \[\GL(A)=\{\alpha:A\to A\mid \alpha \text{ invertible}\}.\]
    For $\alpha\in\GL(A)$, take a piecewise smooth path $\alpha_t$ from $\alpha_0=\id$ to $\alpha_1=\alpha$. The \emph{Fuglede-Kadison determinant} of $\alpha$ is defined by
    \[\ln\Det_\tau \alpha=\int_0^1\Re\tau(\alpha_t^{-1}\alpha'_t)\,\drm t.\]
    which is independent of the path.

    The \emph{determinant line} of $A$ is the real vector space spanned by admissible inner products on $A$, quotient out the subspace generated by 
    \[\langle\ ,\ \rangle'=\Det_\tau(\alpha)^{-1/2}\langle\ ,\ \rangle,\]
    where the admissible inner products are related by \eqref{eq:ip}. This one-dimensional vector space is denoted as $\det A$, or $\det_\tau A$ if we need to specify the trace. It comes with a natural orientation, where positive vectors are represented by $\langle\ ,\ \rangle$. One can think of its elements as volume forms on $A$. In general, there is no canonical isomorphism $\det A\cong \mathbb{R}$. The basic properties of determinant lines are summarized below. For proofs, see \cite{farber1996homological}. 
    \begin{lem}\label{lem:det-iso}
        Let $A,B$ be objects of $\cat{H}(A)$. 
        \begin{enumerate}
            \item There is a canonical isomorphism $\det A\otimes \det B\cong \det(A\oplus B)$. 
            \item In general, if 
            \[0\to H'\xrightarrow{\alpha} H\xrightarrow{\beta} H''\to 0\]
            is a short exact sequence in $\cat{H}(\alg{A})$, then there is an isomorphism $\det H'\otimes\det H''\cong \det H$ determined by $\alpha,\beta$. 
            \item If $f:A\to B$ is an isomorphism in $\cat{H}(\alg{A})$, then 
            \[f_*:\det A\to \det B, \quad \langle\ ,\ \rangle\mapsto \langle f^{-1}(-),f^{-1}(-)\rangle,\]
            is an isomorphism. 
        \end{enumerate}
        The isomorphisms above are all orientation preserving. 
    \end{lem}

    Let $\vir{X}\in\cat{E}(\alg{A})$, $\vir{X}=(A'\xrightarrow{\alpha}A)$. Then the determinant line of $\vir{X}$ is defined as 
    \[\det\vir{X}=\det A\otimes(\det A'/\Ker\alpha)^{-1}.\]
    Here, $V^{-1}=V^*$. It also has a natural orientation. Lemma \ref{lem:det-iso} is naturally generalized to this situation. For a graded object $\vir{X}=\bigoplus_i\vir{X}^i$, define 
    \[\det\vir{X}=\bigotimes_i(\det\vir{X}^i)^{(-1)^i}.\]

    Let $\vir{X}$ be torsion. We may assume that it is given by $A'\xrightarrow{\alpha}A$ where $\alpha$ is injective and has dense image. Fix any admissible inner products $\langle\ ,\ \rangle$ and $\langle\ ,\ \rangle'$ on $A$ and $A'$, we can compute 
    $$(\Det_\tau(\alpha^*\alpha))^{1/2}\langle\ ,\ \rangle(\langle\ ,\ \rangle')^{-1}\in\det\vir{X}.$$
    If $\vir{X}$ satisfies a condition known as $\tau$-triviality \cite{braverman2005l2}, the element is independent of the choice of inner products and one can use it to get a canonical isomorphism $\det\vir{X}\cong\mathbb{R}$. 

    Suppose that $C=(C^\bullet,d)$ is a projective cochain complex. Then $H^\bullet(C)$ is also a graded object in $\cat{E}(\alg{A})$. There is a canonical orientation preserving isomorphism 
    \[\nu_C:\det C\to\det H^\bullet(C)\]
    as in the classical case. 

    \begin{dfn}
        Suppose $C$ is a projective cochain complex and there is a preferred choice $\langle\ ,\ \rangle_i$ on each $C^i$. Denote $\sigma_i\in\det C^i$ to be the element in the determinant line represented by $\langle\,\ \rangle_i$. Define 
        \[\sigma=\prod_{i=0}^n\sigma^{(-1)^i}\in\det C\]
        and the \emph{torsion} of $C$ is 
        \[\rho_C=\nu_C(\sigma)\in\det H^\bullet(C).\]
    \end{dfn}

    We say that $C$ is \emph{of determinant class} if the torsion part of any $H^i(C)$ is $\tau$-trivial, and under this assumption, $H^i(C)\cong H^i_{(2)}(C)$ canonically and we get the usual definition of $L^2$-torsion. If in addition $C$ is \emph{weakly acyclic}, i.e., $H^i_{(2)}(C)=0$ for all $i$, then $\det H^\bullet(C)\cong\mathbb{R}$ canonically, and thus the torsion can be regarded as a real number. It should be emphasized that the torsion depends on the choice of inner products. Note that the inner product on $C^i$ induces an inner product on $H_{\reduce}^i(C)$, and thus trivializes the determinant line. If $C$ is of determinant class, we have 
    \begin{equation}
        \det C\xrightarrow{\nu_C}\det H^\bullet(C)\xrightarrow{\cong} \det H^{\bullet}_{\reduce}(C)\xrightarrow{\cong}\mathbb{R}, \label{eq:detcls}
    \end{equation}
    and thus $\rho_C$ can be viewed as a real number. This coincides with the definition in \cite{lueck2002l2} except that we do not take logarithm and the sign differs because we are using cochain complexes. 
    
    We can view the torsion as a functor $\rho$ that takes a projective cochain complex $C$ with preferred inner products into a pair $(\det H^\bullet(C),\rho_C)$. The source of $\rho$ is the category of projective cochain complexes with preferred inner products. The target of $\rho$ is the category of pointed lines. The advantage of working in the extended category is that $\rho$ is multiplicative, or more generally, it preserves short exact sequences:
    
    Suppose $L,M,N$ are projective cochain complexes and there is a short exact sequence in $\cat{H}(\alg{A})$ 
    \[0\to L\xrightarrow{\alpha}M\xrightarrow{\beta}N\to 0,\]
    i.e., $M$ is a direct sum of $L,N$ as Hilbertian spaces. Then by Property 2 of Lemma \ref{lem:det-iso}, there is a natural isomorphism
    \[\psi_{\alpha,\beta}:\det L\otimes \det N\to \det M.\]
    If $\sigma_L$ and $\sigma_N$ are preferred choices of inner products on $L$ and $N$, then there is a preferred choice of inner products $\sigma_M$ such that
    \[\sigma_M=\psi_{\alpha,\beta}(\sigma_L\otimes\sigma_N).\] 
    Under this assumption, we have 
    \[\rho_M=\rho_L\rho_N,\]
    where we identify $\det H^\bullet(L)\otimes \det H^\bullet(N)$ with $\det H^\bullet(M)$ via isomorphisms $\psi_{\alpha,\beta},\nu_L,\nu_M,\nu_N$. Note that under this identification, if $L,M,N$ are of determinant class and weakly acyclic, then the formula $\rho_M=\rho_L\rho_N$ recovers the usual formula where the torsions are real numbers. To be more precise, we have 

    \begin{lem}\label{lem:split}
        If $L,N$ are projective cochain complexes that are of determinant class and weakly acyclic, and if 
        \[0\to L\xrightarrow{\alpha}M\xrightarrow{\beta}N\to 0,\]
        is exact in $\cat{H}(\alg{A})$, then $M$ is also of determinant class and weakly acyclic; moreover, the isomorphism 
        $$\nu_M\circ\psi_{\alpha,\beta}\circ(\nu_L\otimes\nu_N)^{-1}:\det H^\bullet(L)\otimes \det H^\bullet(N)\to \det H^\bullet(M)$$
        reduces to the multiplication $\mathbb{R}\otimes\mathbb{R}\to\mathbb{R}$. 
    \end{lem}
    \begin{proof}
        The first part is a consequence of \cite[Lemma 3.68]{lueck2002l2}, or one can verify it directly. For the second part, note that $\det H^\bullet(L)=\bigotimes\det H^i(L)^{-1}$, and $H^i(L)=(L^{i-1}\xrightarrow{d}\Ker d|_{L^i})$ is torsion; the similar is hold for $M,N$. We can use $\sigma_L$ to define the inner products on $\Ker d|_{L^i}$ and $L^{i-1}/\Ker d|_{L^{i-1}}$. We can use them to get a canonical element $\sigma'_L\in \det H^\bullet(L)$. If we identify $M=L\oplus N$, then we can write the differential on $M$ as 
        \[d_M=\begin{pmatrix}
            d_L & \ast \\ 
            0 & d_N
        \end{pmatrix}\] 
        The relation 
        \[\Det d_M^*d_M=\Det d_L^*d_L\cdot \Det d_N^*d_N\]
        holds because of the weakly acyclic condition, c.f., \cite[Theorem 3.14]{lueck2002l2}. It follows from a direct calculation that $\nu_M\circ\psi_{\alpha,\beta}\circ(\nu_L\otimes\nu_N)^{-1}$ maps $\sigma'_L\otimes\sigma'_N$ to $\sigma'_M$, and thus the induced map $\mathbb{R}\otimes\mathbb{R}\to\mathbb{R}$ is the multiplication. 
    \end{proof}

    Since the torsion is now defined as an element of the determinant line, to study torsions of different complexes we have to identify their determinant lines in a canonical way. For convenience we introduce the following notion of admissibility. 

    \begin{dfn}
        Let $C_i,D_j$ be projective cochain complexes, and let $a_i,b_j$ be integers. An isomorphism 
        $$\varphi:\bigotimes_i (\det C_i)^{a_i}\to\bigotimes_j(\det D_j)^{b_j}$$
        is called \emph{admissible} if it can be reduced to the canonical map 
        $$\bigotimes_i \mathbb{R}^{a_i}\to\bigotimes_j\mathbb{R}^{b_j}$$
        when each $C_i,D_j$ is of determinant class and weakly acyclic. (The exponents stand for tensor product, not direct sum.)
    \end{dfn}

    Then Lemma \ref{lem:split} is equivalent to say that $\psi_{\alpha,\beta}$ is admissible. 

    \subsection{Combinatorial \texorpdfstring{$L^2$}{L2}-torsion}

    Let $K$ be a connected finite CW complex, with fundamental group $\pi$. Let $(\alg{A},\tau)$ be a finite von Neumann algebra as before. 

    \begin{dfn}
        A (left) \emph{Hilbertian $(\pi,\alg{A})$-bimodule} is a (left) Hilbertian $\alg{A}$-module $H$ with a left $\pi$-action such that every $g\in\pi$ acts as a morphism $L_g:H\to H$ of Hilbertian $\alg{A}$-modules. We say that $H$ is \emph{unimodular} if $\Det_\tau(L_g)=1$ for all $g\in\pi$. 
    \end{dfn}

    Let $C_\bullet(\tilde{K})$ be the cellular chain complex of the universal cover of $K$. The cochain complex 
    \begin{equation}\label{eq:unimod}
        C^\bullet(K;H)=\Hom_{\mathbb{Z}\pi}(C_\bullet(\tilde{K}),H)
    \end{equation}
    is a projective cochain complex in $\cat{H}(\alg{A})$. It defines an isomorphism 
    \[\nu=\nu_{C^\bullet(K;H)}:\det C^\bullet(K;H)\to\det H^\bullet(C^\bullet(K;H)).\]

    If $H$ is unimodular, there is a natural isomorphism 
    \[\psi:(\det H)^{\chi(K)}\to \det C^\bullet(K;H)\]
    that only depends on $K$ and $H$. It follows from the following fact which will be used later. 
    \begin{lem}\label{lem:cell}
        If $c_k$ is the set of $k$-cells of $K$, then
        \[C^k(K;H)\cong\bigoplus_{c_k}H\]
        as Hilbertian $\alg{A}$-modules. The isomorphism depends on the lifting of cells. If $H$ is unimodular, then the isomorphism of determinant lines 
        \[\det C^k(K;H)\cong (\det H)^{\#c_k}\]
        is independent of the lifting and the ordering of the cells. Here, $\#c_k$ is the cardinality of $c_k$. 
    \end{lem}

    \begin{dfn}\label{def:tor}
        If there is a preferred choice of $\sigma\in\det H$, then
        \[\rho(K;H)=\nu(\psi(\sigma^{\chi(K)}))\in\det H^\bullet(K;H)\]
        is called the \emph{combinatorial $L^2$-torsion} of $K$ with coefficients in $H$. 
    \end{dfn}

    If $\alg{A}$ is the group von Neumann algebra of $\pi$ and $H=l^2(\pi)$, then we write 
    \[C^\bullet(K;H)=C^\bullet_{(2)}(K),\quad H^\bullet(C^\bullet(K;H))=H^\bullet_{(2)}(K),\]
    \[\rho(K;H)=\rho_{(2)}(K)\in \det H^{\bullet}_{(2)}(K).\]
    There is a natural choice of $\sigma$ which is the standard $l^2$-inner product, so the torsion is well defined as an element of the determinant line. 

    In fact, we will use a general setting to deduce the sum formula and the product formula. Let $X$ be a finite $\pi$-CW complex where $\pi$ is any discrete group. Let $H$ be a Hilbertian $(\pi,\alg{A})$-bimodule. Then $C_\bullet(X)$ is a $\mathbb{Z}\pi$-module and we set 
    \[C^\bullet(X;H)=\Hom_{\mathbb{Z}\pi}(C_\bullet(X),H).\]
    The above results can be easily generalized to this case. For example, Lemma \ref{lem:cell} holds if we use $c_k$ to denote the set of equivariant $k$-cells. In the isomorphism $\psi$, $\chi(K)$ is replaced by $\chi(\pi\backslash X)=\sum_k (-1)^{\#c_k}$. 

    \section{The sum formula and the product formula}

    \subsection{The sum formula}

    Suppose that $X$ is a finite $\pi$-CW complex where $\pi$ is a discrete group, and that $(\alg{A},\tau)$ is a finite von Neumann algebra. Let $H$ be a Hilbertian $(\pi,\alg{A})$-bimodule with a preferred choice of $\sigma\in\det H$.  

    \begin{prp}\label{prp:sum}
        Suppose that 
        \[\begin{tikzcd}
            X_0 \rar{j_1} \dar[swap]{j_2} & X_1 \dar{i_1} \\ 
            X_2 \rar{i_2} & X 
        \end{tikzcd}\]
        is a $\pi$-pushout diagram, where $X_0,X_1,X_2,X$ are finite $\pi$-CW complexes. Assume that $j_1$ is an inclusion of a subcomplex, and that $j_2$ is cellular. 

        Then there exists a short exact sequence 
        \[0\to C^\bullet(X;H)\to C^\bullet(X_1;H)\oplus C^\bullet(X_2;H)\to C^\bullet(X_0;H)\to 0\]
        in $\cat{H}(\alg{A})$. This defines 
        \[\det H^\bullet(X;H)\otimes \det H^\bullet(X_0;H)\cong \det H^\bullet(X_1;H)\otimes \det H^\bullet(X_2;H).\]
        Under this identification, the torsions are related by 
        \[\rho(X;H)=\rho(X_1;H)\rho(X_2;H)\rho(X_0;H)^{-1}.\]
    \end{prp}

    \begin{proof}
        From the short exact sequence of $\mathbb{Z}\pi$-modules 
        \[0\to C_\bullet(X_0)\to C_\bullet(X_1)\oplus C_\bullet(X_2)\to C_\bullet(X)\to 0,\]
        we obtain a short exact sequence in $\cat{H}(\alg{A})$: 
        \[0\to C^\bullet(X;H)\to C^\bullet(X_1;H)\oplus C^\bullet(X_2;H)\to C^\bullet(X_0;H)\to 0.\]
        Hence,
        \[\det H^\bullet(X;H)\otimes \det H^\bullet(X_0;H)\cong \det H^\bullet(X_1;H)\otimes \det H^\bullet(X_2;H)\]
        by Properties 2 and 3 of Lemma \ref{lem:det-iso}. 

        Using Definition \ref{def:tor}, the required formula follows from the sum formula for Euler characteristic
        \[\chi(\pi\backslash X)=\chi(\pi\backslash X_1)+\chi(\pi\backslash X_2)-\chi(\pi\backslash X_0).\]
    \end{proof}

    Next we indicate how to recover the classical formula in \cite[Theorem 3.93(2)]{lueck2002l2}. 

    \begin{lem}\label{lem:sum}
        The isomorphism in Proposition \ref{prp:sum} is natural, i.e., if three of $X_0,X_1,X_2,X$ are weakly acyclic and of determinant class, then all are weakly acyclic and of determinant class, and the isomorphism 
        \[\det H^\bullet(X;H)\otimes \det H^\bullet(X_0;H)\cong \det H^\bullet(X_1;H)\otimes \det H^\bullet(X_2;H)\]
        can be reduced to the canonical map $\mathbb{R}\otimes\mathbb{R}\to\mathbb{R}\otimes\mathbb{R}$. 
    \end{lem}

    \begin{proof}
        The first part is proved in \cite[Theorem 3.93(2)]{lueck2002l2}. This second part follows from the proof of Proposition \ref{prp:sum} and Lemma \ref{lem:split}. 
    \end{proof}

    \begin{cor}
        Suppose that 
        \[\begin{tikzcd}
            K_0 \rar{j_1} \dar[swap]{j_2} & K_1 \dar{i_1} \\ 
            K_2 \rar{i_2} & K 
        \end{tikzcd}\]
        is a pushout diagram, where $K_0,K_1,K_2,K$ are finite CW complexes. Assume that $j_1$ is an inclusion of a subcomplex, and that $j_2$ is cellular. Assume that $i_1,i_2$ and $i_0=i_1j_1$ induce injections on fundamental groups. 

        Then there exists a short exact sequence 
        \[0\to C^\bullet(K;H)\to C^\bullet(K_1;H)\oplus C^\bullet(K_2;H)\to C^\bullet(K_0;H)\to 0\]
        in $\cat{H}(\alg{A})$. This defines 
        \[\det H^\bullet(K;H)\otimes \det H^\bullet(K_0;H)\cong \det H^\bullet(K_1;H)\otimes \det H^\bullet(K_2;H).\]
        Under this identification, the torsions are related by 
        \[\rho(K;H)=\rho(K_1;H)\rho(K_2;H)\rho(K_0;H)^{-1}.\]
    \end{cor}
    The injectivity condition is necessary to lift the pushout diagram to a $\pi_1(K)$-pushout diagram. 

    \subsection{The product formula}

    Suppose that $(\alg{A}_1,\tau_1)$ and $(\alg{A}_2,\tau_2)$ are finite von Neumann algebras. Let 
    \[(\alg{A},\tau)=(\alg{A}_1\otimes\alg{A}_2,\tau_1\otimes\tau_2)\] 
    be the tensor product. Let $H_i\in\cat{H}(\alg{A}_i)$ and $\alpha_i\in\GL(H_i)$, $i=1,2$. Then we have 

    \begin{lem}\label{lem:det-tensor}
        Let $\dim H_i=\tau_i(\id_{H_i})$ be the von Neumann dimension. Then 
        \[\Det_\tau(\alpha_1\otimes\alpha_2)=(\Det_{\tau_1} \alpha_1)^{\dim H_2}(\Det_{\tau_2} \alpha_2)^{\dim H_1}.\]
    \end{lem}

    \begin{proof}
        Let $\alpha_{i,t}$ be a smooth path from $\id_{H_i}$ to $\alpha_i$. Then 
        \[\alpha_t=\begin{cases}
            \alpha_{1,2t}\otimes\id_{H_2}, & t\in[0,1/2] \\
            \id_{H_1}\otimes \alpha_{2,2t-1}, & t\in[1/2,1]
        \end{cases}\]
        is a piecewise smooth path from $\id_{H}$ to $\alpha_1\otimes\alpha_2$. Computing the Fuglede-Kadison determinant along this path, the result follows easily.
    \end{proof}

    \begin{lem}\label{lem:tp}
        Assume that $H_1,H_2$ both have von Neumann dimension $1$. There is a natural orientation-preserving isomorphism 
        \[\det_\tau H_1\otimes H_2\cong\det_{\tau_1} H_1\otimes\det_{\tau_2} H_2.\]
    \end{lem}
    
    \begin{proof}
        By Lemma \ref{lem:det-tensor}, we get a well defined map 
        \[\det_{\tau_1} H_1\otimes\det_{\tau_2} H_2\to \det_\tau H_1\otimes H_2,\]
        \[\langle\ ,\ \rangle_1\otimes\langle\ ,\ \rangle_2 \mapsto (\langle v\otimes w, v'\otimes w' \rangle=\langle v,v'\rangle_1\langle w,w'\rangle_2).\]
        The function map is clearly nontrivial, and thus is an isomorphism. 
    \end{proof}

    \begin{prp}\label{prp:product}
        Assume that $X_i$ be a finite $\pi_i$-CW complex, $i=1,2$. Let $H_i$ be a unimodular Hilbertian $(\pi_i,\alg{A}_i)$-bimodule with von Neumann dimension $1$. Then there is a natural isomorphism 
        \[\det_\tau C^\bullet(X_1\times X_2;H_1\otimes H_2)\cong(\det_{\tau_1}C^\bullet(X_1;H_1))^{\chi(\pi_2\backslash X_2)}\otimes(\det_{\tau_2}C^\bullet(X_2;H_2))^{\chi(\pi_1\backslash X_1)}.\]
        Under this identification, 
        \[\rho(X_1\times X_2;H_1\otimes H_2)=\rho(X_1;H_1)^{\chi(\pi_2\backslash X_2)}\rho(X_2;H_2)^{\chi(\pi_1\backslash X_1)},\]
        if $\sigma\in\det_\tau H_1\otimes H_2$ is chosen to be $\sigma_1\sigma_2$ where $\sigma_i\in\det_{\tau_i}H_i$. 
    \end{prp}

    \begin{proof}
        Note that $C^\bullet(X_1\times X_2;H_1\otimes H_2)\cong C^\bullet(X_1;H_1)\otimes C^\bullet(X_2;H_2)$. Induction on the number of cells in $X_2$. The basis step is when $X_2=\pi_2\cdot e^i$ is a single equivariant cell. Let $c_k$ be the set of equivariant $k$-cells in $X_1$. In this proof we will write $=$ for any canonical isomorphism. We compute the $n$-th position 
        \[
            [C^\bullet(X_1;H_1)\otimes C^\bullet(e^i;H_2)]^n=\left(\bigoplus_{c_{n-i}} H_1\right)\otimes H_2 = \bigoplus_{c_{n-i}} (H_1\otimes H_2). 
        \]
        Thus, 
        \begin{align*}
            \det_\tau[C^\bullet(X_1;H_1)\otimes C^\bullet(e^i;H_2)]^n&=\bigotimes_{c_{n-i}}\det_\tau H_1\otimes H_2 \\ 
            &=\bigotimes_{c_{n-i}}(\det_{\tau_1}H_1\otimes \det_{\tau_2}H_2)\\
            &=(\det_{\tau_1}H_1)^{\# c_{n-i}}\otimes (\det_{\tau_2}H_2)^{\# c_{n-i}}\\
            &=(\det_{\tau_1}C^{n-i}(X_1;H_1))\otimes (\det_{\tau_2}H_2)^{\# c_{n-i}}.
        \end{align*}
        Taking alternating product, we get 
        \begin{align}\label{eq:prod}
            \det_\tau C^\bullet(X_1;H_1)\otimes C^\bullet(e^i;H_2)&=\bigotimes_n ((\det_{\tau_1}C^{n-i}(X_1;H_1))^{(-1)^n}\otimes (\det_{\tau_2}H_2)^{(-1)^n\# c_{n-i}}) \nonumber \\
            &=(\det_{\tau_1}C^\bullet(X_1;H_1))^{(-1)^i}\otimes (\det_{\tau_2}C^\bullet(e^i;H_2))^{\chi(\pi_1\backslash K_1)}.
        \end{align}
        This proves the basis step. 

        The induction step is similar. Suppose that $X_2$ is obtained from $X_2'$ by attaching a cell $e^i$. Let $H_2'$ be the restriction of $H_2$ to $X_2'$. The restriction of $H_2$ to $e^i$ is still denoted as $H_2$. By Lemma \ref{lem:cell}, we obtain 
        \[C^\bullet(X_2;H_2)=C^\bullet(X_2';H_2')\oplus C^\bullet(e^i;H_2).\]
        We have 
        \[C^\bullet (X_1;H_1)\otimes C^\bullet(X_2;H_2)=(C^\bullet(X_1;H_1)\otimes C^\bullet (X_2';H_2'))\oplus (C^\bullet(X_1;H_1)\otimes C^\bullet(e^i;H_2)).\]
        Now use \eqref{eq:prod} and the induction hypothesis. 
        \begin{align*}
            \det_\tau C^\bullet (X_1;H_1)\otimes C^\bullet(X_2;H_2)&=\begin{aligned}[t]
                ((&\det_{\tau_1}C^\bullet(X_1;H_1))^{\chi(\pi_2\backslash X_2')}\otimes(\det_{\tau_2}C^\bullet(X_2';H_2'))^{\chi(\pi_1\backslash K_1)})\\
                & {} \oplus ((\det_{\tau_1}C^\bullet(X_1;H_1))^{(-1)^i}\otimes (\det_{\tau_2}C^\bullet(e^i;H_2))^{\chi(\pi_1\backslash K_1)})
            \end{aligned}\\
            &=(\det_{\tau_1}C^\bullet(X_1;H_1))^{\chi(\pi_2\backslash K_2)}\otimes(\det_{\tau_2}C^\bullet(X_2;H_2))^{\chi(\pi_1\backslash K_1)}.
        \end{align*}

        Note that the isomorphism of determinant lines is independent of the lifting because of unimodularity. Lastly, using Definition \ref{def:tor}, the required formula follows from the product formula for Euler characteristic 
        \[\chi(\pi_1\times \pi_2\backslash X_1\times X_2)=\chi(\pi_1\backslash X_1)\chi(\pi_2\backslash X_2).\]
    \end{proof}

    \begin{cor}
        Assume that $K_i$ be a finite CW complex with fundamental group $\pi_1(K_i)=\pi_i$, $i=1,2$. Let $H_i$ be a unimodular Hilbertian $(\pi_i,\alg{A}_i)$-bimodule with von Neumann dimension $1$. Then there is a natural isomorphism 
        \[\det_\tau C^\bullet(K_1\times K_2;H_1\times H_2)\cong(\det_{\tau_1}C^\bullet(K_1;H_1))^{\chi(K_2)}\otimes(\det_{\tau_2}C^\bullet(K_2;H_2))^{\chi(K_1)}.\]
        Under this identification, 
        \[\rho(K_1\times K_2;H_1\times H_2)=\rho(K_1;H_1)^{\chi(K_2)}\rho(K_2;H_2)^{\chi(K_1)},\]
        if $\sigma\in\det_\tau H_1\otimes H_2$ is chosen to be $\sigma_1\sigma_2$ where $\sigma_i\in\det_{\tau_i}H_i$. 
    \end{cor}

    The isomorphism in Proposition \ref{prp:product} is admissible, but this is not enough to recover the result in \cite[Theorem 3.93(4)]{lueck2002l2}. Instead, we need a stronger result. Recall that we have chosen $\sigma_i\in\det_{\tau_i}H_i$, and thus we have $\psi(\sigma_i)\in\det_{\tau_i}C^\bullet(X_i;H_i)$ where $\psi$ is defined in \eqref{eq:unimod}. On the other hand, $\sigma=\sigma_1\sigma_2$ defines a preferred inner product on $H_1\otimes H_2$ by Lemma \ref{lem:tp}, and thus we have $\psi(\sigma)\in\det_\tau C^\bullet(X_1\times X_2;H_1\otimes H_2)$. Then the determinant lines can be trivialized as in \eqref{eq:detcls}. The torsion is sent to a real number which is the classical definition. 

    \begin{lem}\label{lem:prod}
        The isomorphism in Proposition \ref{prp:product} is natural in the following sense: If $C^\bullet(X_1;H_1)$ and $C^\bullet(X_2;H_2)$ are of determinant class, then $C^\bullet(X_1;H_1)\otimes C^\bullet(X_2;H_2)$ is of determinant class; moreover, the isomorphism 
        \[\det_\tau C^\bullet(X_1\times X_2;H_1\otimes H_2)\cong(\det_{\tau_1}C^\bullet(X_1;H_1))^{\chi(\pi_2\backslash X_2)}\otimes(\det_{\tau_2}C^\bullet(X_2;H_2))^{\chi(\pi_1\backslash X_1)}.\]
        reduces to the canonical map $\mathbb{R}\to\mathbb{R}^{\chi(\pi_2\backslash X_2)}\otimes\mathbb{R}^{\chi(\pi_1\backslash X_1)}$. 
    \end{lem}
    \begin{proof}
        The first part is proved in \cite[Theorem 3.35(6b)]{lueck2002l2}. To check the second part, we need to check that $\psi(\sigma)$ is mapped to $\psi(\sigma_1)^{\chi(\pi_2\backslash X_2)}\psi(\sigma_2)^{\chi(\pi_1\backslash X_1)}$ under the isomorphism. This follows from the definition of $\psi$, Lemma \ref{lem:tp}, and the inductive proof of Proposition \ref{prp:product}. 
    \end{proof}

    \section{\texorpdfstring{$L^2$}{L2}-torsion of fibrations}

    Let $p:E\to B$ be a simple fibration in the sense of L\"uck-Schick-Thielmann \cite{luck1998torsion}, which means that there is a coherent choice of simple structures on fibers. Denote $F_b=p^{-1}(b)$ and let $F$ be the homotopy type of $F_b$. Assume that $B$ is a connected finite CW complex and choose simple structures $F\to F_b$ on fibers coherently. Then $E$ is equipped with an induced simple structure. Note that if $p$ is a smooth fiber bundle, then it is a simple fibration via a smooth triangulation. A key property of a simple fibration is that over any cell $e$ of $B$, the restriction $p^{-1}(e)$ is homotopy equivalent to the trivial bundle $F\times e$. 

    \begin{thm}\label{thm:main}
        Let $p:E\to B$ be a simple fibration where $B$ is a connected finite CW complex. Fix a basepoint $b\in B$ and assume that $F_b$ has the homotopy type of a connected finite CW complex $F$. Let $i:F_b\to E$ be the inclusion. 

        If $i$ induces an injection on fundamental groups, and if $\chi(F)=0$, then 
        \[\det H^\bullet_{(2)}(E)\cong (\det H^\bullet_{(2)}(F))^{\chi(B)},\]
        and under this identification, 
        \[\rho_{(2)}(E)=\rho_{(2)}(F)^{\chi(B)}.\]
    \end{thm}

    The proof follows from the sum formula and the product formula above, using the same argument as in \cite{lueck2002l2}. We first prove

    \begin{prp}\label{prp:cover}
        Let $p:E\to B$ be a simple fibration as in Theorem \ref{thm:main}. Suppose that $q:\bar E\to E$ is a $\pi$-covering. Let $\overline{F_b}=q^{-1}(F_b)$. Then $\pi$ acts on both $\bar E$ and $\overline{F_b}$. Let $H$ be a unimodular $(\pi,\alg{A})$-bimodule of von Neumann dimension $1$. 

        If $\chi(F)=0$, then
        \[\det H^\bullet(\bar{E};H)\cong(\det H^\bullet(\overline{F_b});H)^{\chi(B)},\]
        and under this identification, 
        \[\rho(\bar E;H)=\rho(\bar F;H)^{\chi(B)}.\]
    \end{prp}

    \begin{proof}
        Induction on the number of cells of $B$. Suppose that $B=B_0\cup_{S^{n-1}}D^n$ and let $E_0=p^{-1}(B_0)$. Let $\overline{E_0}$ be the pullback of $\bar E$. Let $F\to F_b$ be the fixed simple structure and $\bar{F}$ be the pullback of $\bar{E}$. Then we have a $G$-pushout diagram
        \[\begin{tikzcd}
            \bar{F}\times S^{n-1} \rar \dar & \overline{E_0} \dar \\ 
            \bar{F}\times D^n \rar & \overline{E}
        \end{tikzcd}\]
        Applying the sum formula to this, we get 
        \[\det H^\bullet(\overline{E_0};H)\otimes\det H^\bullet(\bar F\times D^n;H)\cong \det H^\bullet(\bar F\times S^{n-1};H)\otimes \det H^\bullet(\bar E;H),\]
        and 
        \[\rho(\bar{E};H)=\rho(\overline{E_0};H)\rho(\bar F\times D^n;H)\rho(\tilde{F}\times S^{n-1};H)^{-1}.\]
        In the product formula, we let $\pi_1=\pi$, $\alg{A}_1=\alg{A}$, $H_1=H$ and let $\pi_2=1$, $\alg{A}_2=1$, $H_2=\mathbb{C}$. Note that $\chi(\pi\backslash\bar F)=\chi(F)$. Thus, 
        \[\det H^\bullet(\bar F\times D^n;H)\cong (\det H^\bullet(\bar F;H))^{\chi(D^n)}\otimes(\det H^\bullet(D^n;\mathbb{C}))^{\chi(F)},\]
        \[\rho(\bar F\times D^n;H)=\rho(\bar F;H)^{\chi(D^n)}\rho(D^n;\mathbb{C})^{\chi(F)};\]
        \[\det H^\bullet(\bar F\times S^{n-1};H)\cong (\det H^\bullet(\bar F;H))^{\chi(S^{n-1})}\otimes(\det H^\bullet(S^{n-1};\mathbb{C}))^{\chi(F)},\]
        \[\rho(\bar F\times S^{n-1};H)=\rho(\bar F;H)^{\chi(S^{n-1})}\rho(S^{n-1};\mathbb{C})^{\chi(F)}.\]

        Now use the assumption that $\chi(F)=0$. The formulas above imply that 
        \[\det H^\bullet(\overline{E_0};H)\cong\det H^\bullet(\bar E;H)\otimes (\det H^\bullet(\bar F;H))^{(-1)^n},\]
        \[\rho(\overline{E_0};H)=\rho(\bar E;H)\rho(\bar F;H)^{(-1)^n},\]
        since $\chi(D^n)-\chi(S^{n-1})=(-1)^n$. 

        The proposition is true when $B=\varnothing$, so by induction we obtain the desired result. 
    \end{proof}

    Next we indicate how to recover the result in \cite[Lemma 3.104]{lueck2002l2}. 

    \begin{lem}
        The isomorphism in Proposition \ref{prp:cover} is admissible, i.e., if $\overline{F_b}$ is weakly acyclic and of determinant class, then so is $\overline{E}$; moreover, the isomorphism 
        \[\det H^\bullet_{(2)}(E)\to (\det H^\bullet_{(2)}(F))^{\chi(B)},\]
        reduces to the canonical map $\mathbb{R}\to\mathbb{R}^{\chi(B)}$. 
    \end{lem}
    \begin{proof}
        The first part is proved in \cite[Lemma 3.104]{lueck2002l2}. Note that the isomorphisms in the proof of Proposition \ref{prp:cover} come from the sum formula and the product formula, which are admissible, c.f., Lemmas \ref{lem:sum} and \ref{lem:prod}. Since in each application of the sum formula, the acyclic assumption is satisfied, and in each application of the product formula, we only need the determinant class condition, we conclude that the isomorphism $\det H^\bullet_{(2)}(E)\to (\det H^\bullet_{(2)}(F))^{\chi(B)}$ is also admissible. 
    \end{proof}

    Lastly, we prove the main theorem. 
        
    \begin{proof}[Proof of Theorem \ref{thm:main}.]
        Apply Proposition \ref{prp:cover} to the case when $q$ is the universal cover, and $H=l^2(\pi_1(E))$. We have $H^\bullet_{(2)}(E)\cong H^\bullet(\bar{E};H)$. The injectivity condition on $i$ implies that $H^\bullet_{(2)}(F)\cong H^\bullet(\bar F;H)$. The Theorem \ref{thm:main} follows easily. 
    \end{proof}

    \section{Open problems}

    The condition $\chi(F)=0$ in Theorem \ref{thm:main} is necessary to obtain the desired formula. Moreover, the injectivity of $i$ in Theorem \ref{thm:main} also seems necessary. In \cite{luck1998torsion,dai2012adiabatic}, a general formula for computing the analytic torsion of the fibration was proved. There is a term coming from the Leray-Serre spectral sequence. We expect that a similar formula holds if we use the equivariant Leray-Serre spectral sequence \cite{moerdijk1993equivariant}. 

    One important aspect of $L^2$-invariants is that they can be approximated by classical invariants. It is interesting to know if the $L^2$-torsion defined in the extended category can also be approximated by the classical counterpart. 
    
    The proofs given in this paper are purely combinatorial. It seems possible to generalize the argument of \cite{dai2012adiabatic} to give an analytic argument. For example, the analytic definition of $L^2$-torsion in the extended abelian category is discussed in \cite{braverman2005l2,zhang2005extended}.
    
    \bibliographystyle{plain}
    \bibliography{fibration}
\end{document}